\documentclass[12pt]{amsart}

\usepackage{amssymb,amsthm,amsmath}
\usepackage{hyperref}
\usepackage[alphabetic]{amsrefs}

\usepackage{enumerate}

\usepackage[margin=1.5in]{geometry}

\newtheorem{theorem}{Theorem}[section]
\newtheorem{corollary}[theorem]{Corollary}
\newtheorem{lemma}[theorem]{Lemma}

\theoremstyle{remark}
\newtheorem{remark}[theorem]{Remark}

\numberwithin{equation}{section}

\newcommand{\N}{\mathbb{N}}

\newcommand{\C}{\mathbb{C}}
\newcommand{\her}{\mathrm{her}}
\newcommand{\Cu}{\mathrm{Cu}}
\newcommand{\Tr}{\mathrm{Tr}}
\newcommand{\U}{\mathrm{U}}

\renewcommand{\epsilon}{\varepsilon}
\renewcommand{\leq}{\leqslant}
\renewcommand{\geq}{\geqslant}
\newcommand{\sa}{\mathrm sa}

\author{Ping Wong Ng}
\address{\hskip-\parindent
Department of Mathematics,
University of Louisiana at Lafayette,
Lafayette, USA.}
\email{png@louisiana.edu}

\author{Leonel Robert}
\address{\hskip-\parindent
Department of Mathematics,
University of Louisiana at Lafayette,
Lafayette, USA.}
\email{lrobert@louisiana.edu}


\title[Determinant on pure C*-algebras]{The kernel of the determinant map on pure C*-algebras}
\begin{document}
\begin{abstract}
In a simple C*-algebra with suitable regularity properties, any unitary or invertible element with de la Harpe--Skandalis determinant zero is a finite product of commutators.
\end{abstract}

\maketitle

\section{Introduction}
The notion of determinant has been, as might be expected, generalized to and investigated  in the realm of operator algebras. Fuglede and Kadison defined a determinant on a type $\mathrm{II}_1$ factor (\cite{Fuglede-Kadison}) and used it to settle fundamental
questions on the trace of the factor. Subsequently, the Fuglede--Kadison determinant  has found manifold applications (see \cite{dlHarpeSurvey} and references therein).
De la Harpe and Skandalis generalized the Fuglede--Kadison  determinant to arbitrary Banach
algebras (\cite{dlHarpe-Skandalis1}), and it  played a role in their investigations on the  basic structure of the
unitary and general linear groups of a C*-algebra
(\cite{dlHarpe-Skandalis1}, \cite{dlHarpe-Skandalis2}, \cite{dlHarpe-Skandalis3}). Again,  the de la Harpe--Skandalis determinant has led to subsequent work (e.g., \cite{thomsen}, \cite{ng}). 
It has been enormously useful
in the Elliott classification program, not least because of its connections with uniqueness 
theorems and the classification of automorphisms (e.g., \cite{Katsura-Matui}, \cite{Lin1}).  

In \cite{dlHarpe-Skandalis2}, de la Harpe and Skandalis show that every invertible or unitary in the kernel of the determinant map of a simple unital AF C*--algebra is expressible as a finite product of commutators. This type of result had been previously considered for the Fuglede--Kadison determinant (\cite{Fack-dlHarpe}). Thomsen extended de la Harpe and Skandalis's result to certain simple slow dimension growth AH C*-algebras \cite{thomsen}. The first named author extended further these results to  simple $\mathcal Z$-stable TAI  and simple real rank zero C*-algebras with strict comparison (\cite{ng}). A common characteristic in these results is the abundance of projections in the C*-algebra.
In this paper, we generalize these previous results to a large class of simple C*-algebras,
including those considered by the Elliott classification program, as well as non-nuclear examples.


Let us state the main result of the paper.
Let $A$ be a C*-algebra. Let $\U^0(A)$ and $\mathrm{GL}^0(A)$ denote the connected components  of the unitary and general linear groups of $A$, respectively.
Let $\Delta_{\Tr}$ denote the universal de la Harpe-Skandalis determinant  defined on $\mathrm{GL}^0(A)$ (see Section 2 below and \cite{dlHarpe-Skandalis1}). 
Throughout the paper we have focused mostly on the restriction of  $\Delta_{\Tr}$ to $\U^0(A)$. However, our methods give
analogous results for the kernel $\Delta_{\Tr}$ on $\mathrm{GL}^0(A)$ (see Theorem \ref{GLcase}). 

We say that the C*-algebra  $A$ is  pure if the Cuntz semigroup of $A$ is almost unperforated and almost divisible (\cite[Definition 2.6]{winter}). C*-algebras that tensorially absorb the Jiang-Su C*-algebra are pure (\cite{rordam}), but pure C*-algebras may even be tensorially prime (e.g.,  the reduced C*-algebra of the free group on infinitely many
generators $C_r^*(\mathbb F_{\infty})$; see \cite[Proposition 6.3.2]{nccw}).

\begin{theorem}\label{manyexp}
Let $A$ be  a separable, simple, pure C*-algebra with stable rank 1 and such that every 2-quasitracial state on $A$ is a trace. Let $u\in \U^0(A)$ be such that $\Delta_{\Tr}(u)=0$ and $u$ is a product of $k$ exponentials; i.e., $u=\prod_{j=1}^k e^{i h_j}$ with $h_j\in A_{\sa}$ for all $j=1,\dots,k$. Then there exist
$7k+29$ unitaries $v_i,w_i\in \U^0(A)$, with $i=1,\dots,7k+29$, such that
\[
u=\prod_{i=1}^{7k+29} (v_i,w_i).
\]
\end{theorem}

A simple corollary of this theorem  is that for a C*-algebra $A$ as in the theorem we have
\[
\ker \Delta_{\Tr}\cap \U^0(A)=\mathrm{DU}^0(A)=\mathrm{DU}(A),
\]
(Corollary \ref{Aug}). 
If we additionally know that the C*-algebra $A$ has finite exponential rank, then Theorem \ref{manyexp} gives a uniform bound on the number of commutators. Lin has shown in \cite{Lin2} that unital, $\mathcal Z$-stable,  rationally TAI C*-algebras have exponential rank  $1+\epsilon$,  so a uniform bound (of 43 commutators) is obtained in this case.

Another motivation for the results of this paper 
is to test
out, extend, and provide a new forum for recent techniques from classification theory.
In the past five years, the Elliott classification program has made great strides, resulting
in many new tools and ideas which should have wide applicability (see \cite{bulletin}).  Three regularity properties 
have come to the foreground: tensorial absorption of the Jiang-Su algebra, nuclear dimension, and pureness (for simple C*-algebras, the weakest of the three).  Pure C*-algebras   form a class of C*-algebras with a well behaved  ``comparison theory" -- in analogy with the comparison theory of factors, and  in the spirit of Blackadar's paper \cite{blackadar2}. 
An example of a pure C*-algebra to which Theorem \ref{manyexp} applies is
$C^*_r(\mathbb{F}_{\infty})$  -- a non-nuclear tensorially prime C*-algebra of much interest.
Along the way towards the proof of Theorem \ref{manyexp}, we also investigate the kernel of the determinant map  on   C*-algebras of nuclear dimension at most one.

Let us discuss briefly the  methods used to   prove Theorem \ref{manyexp}: We use that it is possible to find a special simple C*-algebra $B$,  expressible as an inductive limit of 1-dimensional NCCW complexes, inside a C*-algebra $A$ satisfying the hypotheses of Theorem \ref{manyexp}.
This relies on the classification result from \cite{nccw}. The approximation of unitaries in $\U^0(A)\cap\ker \Delta_{\Tr}$ by products of commutators is then reduced to unitaries in $\U^0(B)$. To approximate unitaries in $\U^0(B)\cap\ker \Delta_{\Tr}$ by products of commutators, we exploit the fact that it  has nuclear dimension at most one.
Finally,  in order to turn the approximation by products commutators into an exact product of commutators, we use a technique originally used by Fack for  AF C*-algebras(\cite{fack}), and which we have adapted to 
a setting without abundance of projections.
  

\section{Definitions, notation, preliminary results}
Let $A$ be a C*-algebra. We will denote by $A_{\sa}$ the subset of selfadjoint elements of $A$ and by $A_+$ the positive elements.

Given $a\in A_+$, we let $\her(a)$ denote the hereditary C*-subalgebra generated by $a$; i.e., $\overline{aAa}$. Some times we will write $\her^A(e)$ if the ambient C*-algebra needs to be emphasized.

The unitization of $A$ will be  denoted by $A^\sim$. The unitary group of $A$ will be  denoted by $\U(A)$ and the connected component of the identity by $\U^0(A)$. We will often deal with unitary groups of non-unital C*-algebras. In this case we will let $\U^0(A)$ stand for the kernel of the group homomorphism $\U^0(A^\sim)\to U(\C)$ induced by the quotient map $A^\sim \to \C$ (i.e. unitaries of the form $1+x$, with $x\in A$).

If $u$ and $v$ are elements in $\U^0(A)$ (or more generally a group) we denote by $(u,v)$ the commutator $uvu^{-1}v^{-1}$.

Let $A_q=A/\overline{[A,A]}$ and let us denote by $\Tr\colon A\to A_q$
the quotient map (which we call the universal trace on $A$). By a theorem of Cuntz and Pedersen,
if $h\in A_{\sa}$ and $\Tr(h)=0$ then $h=\sum_{i=1}^\infty [x_i^*,x_i]$ for some $x_i\in A$ (see \cite{cuntz-pedersen}). We will often write
$a\sim_{\Tr} b$ meaning that $\Tr(a-b)=0$.
 We extend $\Tr$ to $\bigcup_{n=1}^\infty M_n(A)$
by setting $\Tr((a_{i,j})_{i,j=1}^n)=\mathrm{Tr}(\sum_{i=1}^n a_{i,i})$.

We will need the following lemma:
\begin{lemma}\label{liftsofTr0}
Let $A$ be a C*-algebra and $I$ a closed two-sided ideal of $A$. 
\begin{enumerate}[(i)]
    \item For any  $h\in (A/I)_{\sa}$  such that $h\sim_{\Tr}0$ and $\epsilon>0$ there exists a lift $h'\in A_{\sa}$ of $h$ such that $h'\sim_{\Tr}0$ and
$\|h'\|\leq (1+\epsilon)\|h\|$. 

\item For any $a,b\in A/I$ such that $a\sim_{\Tr}b$ and $a'\in A$, lift of $a$, there exists $b'\in A$, lift of $b$, such that $a'\sim_{\Tr} b'$.
\end{enumerate}
\end{lemma}

\begin{proof}
(i) We first look at the case of that $h$ is finite sum of commutators. Say $h=\sum_{i=1}^n [x_i^*,x_i]$.
By lifting each commutator $x_i$ to some $\tilde x_i\in A$, we find $\tilde h=\sum_{i=1}^n [\tilde x_i^*,\tilde x_i]$, lift of $h$, such that
$\tilde h\sim 0$. Let $(e_\lambda)_{\lambda}$ be an approximately central approximate unit of the ideal $I$. 
Then  $(1-e_\lambda)\tilde h(1-e_{\lambda})$ is also lift of $h$ and $\|(1-e_\lambda)\tilde h(1-e_{\lambda})\|\to \|h\|$.
These elements are almost sums commutator: 
\[
(1-e_\lambda)\tilde h(1-e_{\lambda})-\sum_{i=1}^n [((1-e_\lambda)\tilde x_i)^*,(1-e_{\lambda})\tilde x_i]\to 0.
\] 
So for some $\lambda$ we find that $h'=\sum_{i=1}^n [((1-e_\lambda)\tilde x_i)^*,(1-e_{\lambda})\tilde x_i]$ is the desired lift of $h$.

In the general case, we can write $h=\sum_{n=1}^\infty h_i$, where $h_i$ is a finite sum of commutators for all $i$, $\|h_1\|\leq (1+\frac \epsilon 2)\|h\|$ and $\|h_i\|\leq \frac{1}{2^i}$ for $i\geq 2$. We then choose lifts for each of the $h_i$'s and we are done.

(ii) Let $c=b-a$. By (i)  we can find $c'\in A$, lift of $c$, such that $c'\sim_{\Tr}0$. Then $b'=a'+c'$ is the desired lift of $b$.
\end{proof}

We will denote by $\Delta_{\Tr}\colon \mathrm{GL}^0_\infty(A)\to A_q/\Tr(\mathrm K_0(A))$ the de la Harpe-Skandalis determinant. This is
the group homomorphism such that $\Delta_{\Tr}(e^{2\pi iy})$ is equal to the equivalence class of $\mathrm{Tr}(y)$ in $A_q/\Tr(\mathrm K_0(A))$ for all $y\in M_\infty(A)$.
We will, however, be mostly concerned with the restriction of $\Delta_{\Tr}$
to $\U^0(A)$, and in particular the kernel of this restriction.

The following lemmas will be used repeatedly throughout the paper:

\begin{lemma}\label{lem:almostcommute}
Let $M>0$ and $\epsilon>0$. There exists $\delta>0$ such that if $a,b\in A_{\sa}$ are such that $\|a\|,\|b\|\leq M$ and $\|[a,b]\|<\delta$ then
\[
e^{i(a+b)}=e^{ia}e^{ib}e^{ic},\]
where $c\sim_{\Tr}0$ and $\|c\|<\epsilon$.
\end{lemma}
\begin{proof}
For each $r>0$ there exists  $\delta>0$ such that $\|[a,b]\|<\delta$ and
$\|a\|,\|b\|<M$ imply $\|e^{a+b}e^{-a}e^{-b}-1\|<r$.
(Assume, to the contrary, that no such $\delta_0$ exists.
Then 
for each $n \geq 1$ there exist a C*-algebra $A_n$
and $a_n, b_n \in A_n$
with
$\| a_n \|, \| b_n \| \leq M$, $\| [a_n, b_n] \| < 1/n$,
and
$\| e^{a_n + b_n} e^{-a_n} e^{-b_n}-1 \| \geq r$.
Consider the C*-algebra
$A = \prod_{n=1}^{\infty} A_n/ \bigoplus_{n=1}^{\infty}
A_n$ and the elements $a,b\in A$ that lift to
$(a_n)_{n=1}^{\infty}$ and
$(b_n)_{n=1}^{\infty}$
 respectively.
Then $a$ and $b$ commute, and so $e^{a+b}
= e^{a} e^{b}$. Hence
$\| e^{a_n + b_n}  e^{-a_n} e^{-b_n}-1 \| \rightarrow 0$,
 which contradicts that
$\| e^{a_n + b_n}  e^{-a_n} e^{-b_n}-1 \| \geq r$ for all $n\in \N$.)
So let us choose $0<\delta$  such  that $\|e^{i(a+b)}e^{-ia}e^{-ib}-1\|<1$, and furthermore
$\|\mathrm{Log}(e^{i(a+b)}e^{-ia}e^{-ib})\|<\epsilon$, whenever $a,b\in A_{\sa}$ are such that $\|[a,b]\|<\delta$ and $\|a\|,\|b\|\leq M$.
Set $\mathrm{Log}(e^{i(a+b)}e^{-ia}e^{-ib})=c$,
  so that $e^{i(a+b)}=e^{ia}e^{ib}e^{ic}$. It remains to show that $c\sim_{\Tr}0$.
Let $\eta(t)=e^{it(a+b)}e^{-ita}e^{-itb}$ for $t\in [0,2\pi]$. Then
$\|\eta(t)-1\|<1$ for all $t$.  So
\[
\Delta_\Tr(\eta)=\Tr(\mathrm{Log}(\eta(1)))=\frac{1}{2\pi}\Tr(c),
\]
by \cite[Lemma 1]{dlHarpe-Skandalis1}.
On the other hand, by \cite[Lemma 1]{dlHarpe-Skandalis1} again,
\begin{align*}
\Delta_\Tr(\eta) &=\Delta_\Tr(e^{it(a+b)})-\Delta_\Tr(e^{ita})-\Delta_\Tr(e^{itb})\\
&=\frac{1}{2\pi}\Tr(a+b)-\frac{1}{2\pi}\Tr(a)-\frac{1}{2\pi}\Tr(b)\\
&=0.\qedhere
\end{align*}
\end{proof}

\begin{lemma}
For any $\epsilon>0$ and $h\in A_{\sa}$ there exists $\delta>0$ such that if $\|h'-h\|<\delta$ then
$e^{ih}=e^{ih'}e^{ic}$, where $c\sim_{\Tr}h-h'$ and $\|c\|<\epsilon$.
\end{lemma}
\begin{proof}
By the previous lemma, we can choose $\delta$ such that $e^{ih}=e^{h'}e^{i(h-h')}e^{ic'}$, with
$c'\sim_{\Tr}0$ and $\|c'\|$ small enough. We can do this so that $e^{i(h-h')}e^{ic'}=e^{ic}$, where
$c\sim_{\Tr}h-h'$ and $\|c\|<\epsilon$.
\end{proof}

\begin{lemma}\label{lem:x2}
Let $x\in A$ be  such that $x^2=0$. The following are true:
\begin{enumerate}[(i)]
\item
For each $\epsilon>0$ there exist $v,w\in \U^0(A)$ such that
\[
e^{i [x^*,x]}=(v,w)e^{i c},
\]
where $c\sim_{\Tr} 0$ and $\|c\|<\epsilon$.

\item
If $\|x\|< \sqrt{\frac {\pi} {2}}$  then there exist $v,w\in \U^0(A)$ such that
\[
e^{ i[x^*,x]}=(v,w),
\]
with  $\|v-1\|,\|w-1\|\leq \|e^{ix^*x}-1\|^{\frac 1 2}$.
\end{enumerate}
\end{lemma}
\begin{proof}
There exists a homomorphism $\phi\colon M_2(C_0(0,\|x\|])\to A$
such that
\[
\begin{pmatrix}
0 & t\\
0  & 0
\end{pmatrix}\stackrel{\phi}{\longmapsto} x
\]
Thus, it suffices to prove the lemma in $B=M_2(C_0(0,\|x\|])$ for  $x=\bigl(\begin{smallmatrix} 0 & t\\0 &0\end{smallmatrix}\bigr)$.
Now, since $B$ has stable rank one,  $x^*x $ and $xx^*$ are approximately unitarily equivalent in $B^\sim$. Let us choose $u\in \U^0(B)$ (notice that $B^\sim$ has connected unitary group) such that
$u^*(x^*x)u$ is close enough to $xx^*$  so that in turn
$e^{iu^*(x^*x)u}=e^{ic}e^{ixx^*}$, with
 $c\sim_{\Tr}0$ and $\|c\|<\epsilon$ (applying the previous lemma).  Then we find that
\[
e^{i[x^*,x]}=e^{ix^*x}e^{-ixx^*}=(e^{ix^*x},u)e^{ic}.
\]

(ii) This  follows from \cite[Lemma 5.13]{dlHarpe-Skandalis2}.
\end{proof}

\section{C*-algebras of nuclear dimension at most one.}
Here we investigate the problem of approximating unitaries $u\in \ker \Delta_{\Tr}\cap \U^0(A)$ by  products of commutators, when the C*-algebra $A$ has nuclear dimension at most 1. Besides controlling the number of commutators, we pay attention to having multiplicative errors of the form $e^{ic}$, with $c\sim_{\Tr}0$. Theorems \ref{nuc1exp} and \ref{thm:prodexp} will be used, later on, in the proof of Theorem  \ref{manyexp}.

The reader is referred to \cite{winter-zacharias} for the definition and basic results on the nuclear dimension of C*-algebras.

Let $A$ be a separable C*-algebra of nuclear dimension at most $m\in \N$. By \cite[Proposition 3.2]{winter-zacharias},
there exist completely positive contractive (c.p.c.) order zero  maps
\begin{align}\label{nuc1factorization}
A\stackrel{\psi_i}{\longrightarrow} N_i\stackrel{\phi_i}{\longrightarrow}A_\infty,
\end{align}
with $i=0,\dots,m$, such that
\[
\iota=\sum_{i=0}^m \phi_i\psi_i,
\]
where $A_\infty=\prod_{n=1}^\infty A/\bigoplus_{n=1}^\infty A$, $\iota\colon A\to A_\infty$  is the inclusion of $A$ as constant sequences, and
\[
N_i=\prod_{n=1}^\infty F_{i,n}\Big/\bigoplus_{n=1}^\infty F_{i,n},
\]
with $F_{i,n}$ finite dimensional C*-algebras for all $i=0,\dots,m$ and $n\in \N$.
Furthermore, for each $i=0,\dots,m$ there exists $e_i\in A'\cap A_\infty$ such that
$\phi_i\psi_i(a)=e_ia$ for all $a\in A$.
We  shall call $\{e_i\}_{i=0}^m$ the partition of unity associated to the decomposition $(\psi_i,N_i,\phi_i)_{i=0}^m$. (In the case that $A$ is unital,
$e_i=\phi_i\psi_i(1)$, for all $i=0,\dots, m$.)
In general, the elements of the partition of unity do not commute with each other.
Nevertheless,  if $A$ is unital and has  nuclear dimension at most one, then   $e_1=1-e_0$ and so $e_0$ and $e_1$ do commute.

\begin{lemma}\label{ultrafinite}
Let $N=\prod_{n=1}^\infty F_n/\bigoplus_{n=1}^\infty F_n$, where each $F_n$ is a finite dimensional C*-algebra. Let $h\in N_{\sa}$ be  such that $h\sim_{\Tr} 0$.
Then there exist $h_1,h_2\in N_{sa}$ such that
\begin{enumerate}[(i)]
    \item $h=h_1+h_2$,
    \item $[h_1,h_2]=0$,
    \item $h_i=[x_i^*,x_i]$   for some $x_i\in N$ such that $x_i^2=0$ and $\|x_i\|\leq \|h\|^{\frac 1 2}$, for $i=1,2$.
\end{enumerate}
\end{lemma}
\begin{proof}
By \cite[Lemma 3.1]{robert-commutators}, we can lift $h$ to $(h_n)_n\in \prod_{n=1}^\infty F_n$ such that $h_n\sim_{\Tr}0$ for all $n\in \N$. We can then prove  the lemma entrywise; i.e., for a finite dimensional C*-algebra. This case, in turn, reduces at once to the full matrix algebra case. So let us suppose that $h\in M_n(\C)_{\sa}$ is such that $\Tr(h)=0$. Conjugating by a unitary, let us further reduce ourselves to the case that  $h$ is a diagonal matrix: $h=\mathrm{diag}(\lambda_1,\dots,\lambda_n)$.
Finally, conjugating by a permutation, let us also assume that the diagonal entries of $h$ are arranged in such a way
that
\[
\Big|\sum_{i=1}^k \lambda_i\Big|\leq \max_{1\leq i\leq k} |\lambda_i|
\]
for all $k=1,\dots,n$.
Let us now proceed, as in the proof of   \cite[Lemma 5.3]{dlHarpe-Skandalis2},
by setting $\mu_k=\sum_{i=1}^k \lambda_i$ for all $k=1,\dots,n$, so that
\[
h=\mathrm{diag}(\mu_1,-\mu_1,\mu_3,-\mu_3\dots)+\mathrm{diag}(\mu_2,-\mu_2,\mu_4,-\mu_4,\dots).
\]
The two diagonal matrices on the right hand side are of the form $[x^*,x]$ with $x^2=0$ (this follows from the  $2\times 2$ case).
\end{proof}

\begin{theorem}\label{nuc1}
Let $A$ be a C*-algebra of nuclear dimension at most 1. Let $h\in A_{\mathrm sa}$ be  such that $h\sim_{\mathrm{Tr}}0$. Then for each $\epsilon>0$
there exist $h_i\in A_{\sa}$, with $i=1,2,3,4$ such that
\begin{enumerate}[(i)]
\item
$h_1+h_2+h_3+h_4\approx_\epsilon h$,
\item
$\|[h_1,h_2]\|<\epsilon$, $\|[h_3,h_4]\|<\epsilon$, and $\|[h_1+h_2,h_3+h_4]\|<\epsilon$,
\item
$h_i=[x_i^*,x_i]$ where $x_i^2=0$  and $\|x_i\|^2\leq \|h\|$ for all $i=1,2,3,4$.
\end{enumerate}
\end{theorem}

\begin{proof}
Let $A$ be as in the statement. We can reduce ourselves to the case that $A$ is unital. For suppose that, having proved the theorem in the unital case, we have that $A$ is not unital. Then passing to $A^\sim$ and finding $x_1,x_2,x_3,x_4\in A^\sim$ that satisfy (iii) of the statement of the theorem, we see from $x_i^2=0$ that these elements are in fact in $A$.   Also, by \cite[Proposition 2.6]{winter-zacharias}, there exists a separable C*-algebra $B\subseteq A$ of nuclear dimension at most one, and such that $h\in B$ and $h\sim_{\Tr}0$ in $B$. Hence, we can assume that $A$ is separable.

So let us assume that $A$ is unital and separable.
Let $e_0,e_1\in A_\infty$ be a partition of the unit
associated to a decomposition $(\psi_i,N_i,\phi_i\mid i=0,1)$ (as in \eqref{nuc1factorization}), so that
$\phi_i\psi_i(h)=e_i h$ for $i=0,1$. Let us fix $i=0,1$. By Lemma \ref{ultrafinite} applied to $\psi_i(h)\in N_i$, there exist selfadjoint elements
$\tilde h_{i,0},\tilde h_{i,1}\in N_i$ such that
\begin{enumerate}
\item
$\psi_i(h)=\tilde h_{i,0}+\tilde h_{i,1}$,
\item
$[\tilde h_{i,0},\tilde h_{i,1}]=0$,
\item
$\tilde h_{i,j}=[\tilde x_{i,j}^*,\tilde x_{i,j}]$, with $\tilde x_{i,j}^2=0$ and $\|x_{i,j}\|^2\leq \|\tilde h_{i,j}\|$.
\end{enumerate}
Now composing with $\phi_0$ and $\phi_1$ we get elements  $h_{i,j}:=\phi_0(\tilde h_{i,j})\in A_\infty$ that commute (since order zero maps preserve commutation).
Let $\pi^{\phi_i}\colon N_i\otimes \C_0((0,1])\to A_\infty$ be the homomorphism such that $\phi_i(a)=\pi^{\phi_i}(a\otimes t)$ for all $a\in N_i$ (see \cite{winter-zacharias}).
 Then $x_{i,j}=\pi^{\phi_i}(\tilde x_{i,j}\otimes t^{1/2})\in A_\infty$ satisfies that
$h_{i,j}=[x_{i,j}^*,x_{i,j}]$ and $x_{i,j}^2=0$.
Finally, we can lift the relations $x_{i,j}^2=0$ and $\|x_{i,j}\|^2\leq \|h\|$ to $\prod_{n=1}^\infty A$, which proves the theorem.
\end{proof}

We get the following consequence of Theorem \ref{nuc1}:
\begin{theorem}\label{nuc1exp}
Let $A$ be a C*-algebra of nuclear dimension at most 1. Let $h\in A_{\sa}$ be such that $h\sim_{\mathrm{Tr}}0$. Let $\epsilon>0$. Then
\begin{enumerate}[(i)]
\item
there exist
unitaries $v_j,w_j\in \U^0(A)$, with $j=1,2,3,4$, and a selfadjoint $c\in A_{\sa}$ such that
\[
e^{ ih}=\prod_{j=1}^4(v_j,w_j)\cdot e^{ic},
\]
where $c\sim_{\Tr}0$ and $\|c\|<\epsilon$.
\item
If $\|h\|<\frac \pi 2$ then the  unitaries from part (i) may be chosen to additionally satisfy that
$\|v_j-1\|,\|w_j-1\|\leq \|e^{ih}-1\|^{\frac 1 2}$ for $j=1,2,3,4$.
\end{enumerate}
\end{theorem}
\begin{proof}
Let $h_{j,n}\in A_{\sa}$ and $x_{j,n}\in A$, with $j=1,2,3,4$ and $n\in \N$  be elements that satisfy conditions Theorem \ref{nuc1} (i)-(iii) for $\epsilon=\frac 1 n$. Then, for $n$ large enough
\begin{align*}
e^{ih} &=e^{i(h_{1,n}+h_{2,n}+h_{3,n}+h_{4,n})}e^{ic_n}\\
&=e^{i(h_{1,n}+h_{2,n})}e^{ i(h_{3,n}+h_{4,n})}e^{ic_n'}\\
&= e^{ih_{1,n}} e^{i h_{2,n}} e^{ih_{3,n}}e^{i h_{4,n}}e^{ic_n''},
\end{align*}
where $c_n\sim_{\Tr}c_n'\sim_{\Tr}c_n''\sim_{\Tr}0$ and $c_n,c_n',c_n''\to 0$.
Here we have used the asymptotic commutation relations from Theorem \ref{nuc1} (ii)
in combination with Lemma \ref{lem:almostcommute}, and  some  algebraic manipulations with exponentials.

To prove (i), we use Lemma \ref{lem:x2} (i) to find for each  $j=1,2,3,4$ unitaries $v_{j,n},w_{j,n}\in \U^0(A)$ such that
\[
e^{ih_{j,n}}=e^{i[(x_{j,n})^*,x_{j,n}]}=(v_{j,n},w_{j,n})e^{id_{j,n}},
\]
where $d_{j,n}\sim_{\Tr}0$ and  $d_{j,n}\to 0$ as $n\to\infty$.
To prove (ii), we use Lemma \ref{lem:x2} (ii) to ensure that $v_{j,n}$ and $w_{j,n}$ can be chosen such that
\[
\|v_{j,n}-1\|,\|w_{j,n}-1\|\leq \|e^{ih_{j,n}}-1\|^{\frac 1 2}\leq  \|e^{ih}-1\|^{\frac 1 2},
\]
for all $j=1,2,3,4$ and   $n\in \N$.
\end{proof}

We now investigate products of exponentials in   a C*-algebra of nuclear dimension at most one.

\begin{theorem}\label{thm:prodexp}
Let $A$ be a C*-algebra of nuclear dimension at most $1$. Let $u\in \U^0(A)$ be of the form
$u=e^{i h_1}e^{ih_2}\cdots e^{i h_k}$,
where $h_j\in A_{\sa}$ for $j=1,\dots,k$ and
\[
\sum_{j=1}^kh_j\sim_{\mathrm{Tr}}0.
\]
Then for each $\epsilon>0$
there exist $v_j,w_j\in \U^0(A)$, with $j=1,\dots,5k-1$ and $c\in A_{\sa}$ such that
\[
u=\prod_{j=1}^{5k-1}(v_j,w_j)\cdot e^{ic},
\]
where $c\sim_{\Tr}0$ and $\|c\|<\epsilon$.
\end{theorem}
Before proving this theorem we need a couple of lemmas:
\begin{lemma}\label{Ncentral}
Let $N=\prod_{n=1}^\infty F_n/\bigoplus_{n=1}^\infty F_n$, where $F_n$ is a finite dimensional C*-algebra for all $n$. Let $h\in N_{\sa}$.
Then there exist $q,r\in N_{\sa}$ such that $h=q+r$, $q\sim_{\mathrm Tr}0$, $r$ is in the center of $N$,
and $\|q\|\leq 2\|h\|$, $\|r\|\leq \|h\|$.
\end{lemma}

\begin{lemma}\label{groupcommutators}
Let $a_1,\dots,a_k$ and $b_1,\dots,b_k$ be element of a group. We have
\[
\prod_{j=1}^k a_jb_j=\prod_{j=1}^k a_j\cdot \prod_{j=0}^{k-1} b_{k-j}\cdot \prod_{j=1}^{k-1}(v_j,w_j),
\]
for some group elements $v_1,w_1,\dots,v_k,w_k$
\end{lemma}
\begin{proof}
Straightforward induction.
\end{proof}
\begin{proof}[Proof of Theorem \ref{thm:prodexp}]
We will assume  that $A$ is separable and unital (as in the proof of Theorem \ref{nuc1}, the general case  reduces to this case).

Let $(\psi_i,N_i,\phi_i\mid i=0,1)$ be a factorization of $\iota\colon A\to A_\infty$ coming from the nuclear dimension of $A$ (as in \eqref{nuc1factorization}). Let $e_i=\phi_i\psi_i(1)$, with $i=0,1$, be the partition of unity associated to it (so
$e_0,e_1\in A'\cap A_\infty$ and $e_0+e_1=1$).

Let $h_1,\dots,h_k\in A_{\sa}$ be selfadjoint elements as in the statement of the theorem.
Since $\sum_{j=1}^k h_j\sim_{\Tr}0$ and the relation $\sim_{\Tr}$ is preserved by order zero maps, we have $\sum_{j=1}^k \psi_i(h_j)\sim_{\mathrm{Tr}}0$ for $i=0,1$.
By Lemma \ref{Ncentral} applied to each of the $\psi_i(h_j)$s, we have
\[
\psi_i(h_j)=q_{j,i}+r_{j,i},
\]
where $q_{j,i}\sim_{\Tr}0$ and $r_{j,i}\in N_i$ is central. Notice that  $\sum_{j=1}^k r_{j,i}\sim_{\mathrm{Tr}}0$.
By Lemma \ref{ultrafinite} and the fact that the maps $\phi_0$ and $\phi_1$ are  order zero maps,  there exist
$x_{j,i,0},x_{j,i,1},y_{i,0},y_{i,1}\in A_\infty$, for $i=0,1$ and $j=2,\dots,k$ such that
\begin{align*}
\phi_i(\psi_i(h_j))=[x_{j,i,0}^*,x_{j,i,0}]+[x_{j,i,1}^*,x_{j,i,1}]+\phi_i(r_{i,j})
\end{align*}
for $i=0,1$ and $j=2,\dots,k$,
\[
\phi_i(\psi_i(h_1))+\sum_{j=2}^k \phi_i(r_{j,i})=[y_{i,0}^*,y_{i,0}]+[y_{i,1}^*,y_{i,1}],\]
for $i=0,1$, and
\[
x_{j,i,0}^2=x_{j,i,1}^2=y_{i,0}^2=y_{i,1}^2=0,
\]
for $i=0,1$ and $j=2,\dots,k$.

The relations $x_{j,i,0}^2=x_{j,i,1}^2=y_{i,0}^2=y_{i,1}^2=0$ can be lifted to $\prod_{n=1}^\infty A$, thus obtaining $(x_{j,i,0}^{(n)})_{n=1}^\infty$, $(x_{j,i,1}^{(n)})_{n=1}^\infty$, and $(y_{j,i}^{(n)})_{n=1}^\infty$.
Next, aided by Lemma \ref{liftsofTr0}, we lift the $\phi_i(r_{j,i})$s to selfadjoint elements $(z_{j,i}^{(n)})_n\in \prod_{n=1}^\infty A$, taking care that the following relations hold for all $n\in \N$:
\[
\sum_{j=1}^kz_{j,0}^{(n)}\sim_{\Tr}0\hbox{ and }
z_{j,0}^{(n)}+z_{j,1}^{(n)}\sim_{\Tr} h_j
\]
for $j=1,\dots,k$. (Notice that $\phi_0(r_{j,0})+\phi_1(r_{j,1})\sim_{\Tr}h_j$ in $A_\infty$.)

For each $j=2,\dots,k$, $i=0,1$, and $n\in \N$, let us set
\begin{align*}
h_{j,i}^{(n)} &:=\sum_{k=0,1}[(x_{j,i,k}^{(n)})^*,x_{j,i,k}^{(n)}]+z_{j,i}^{(n)},\\
h_{1,i}^{(n)} &:=[(y_{i,0}^{(n)})^*,y_{i,0}^{(n)}]+[(y_{i,1}^{(n)})^*,y_{i,1}^{(n)}]-\sum_{j=2}^k z_{j,i}^{(n)}.
\end{align*}
Then $(h_{j,0}^{(n)})_n$ and $(h_{j,1}^{(n)})_n$ are lifts of $\phi_0(\psi_0(h_j))$ and $\phi_1(\psi_1(h_j))$ respectively for all $j=1,\dots,k$. In particular, they asymptotically commute. They also satisfy that $h_{j,0}^{(n)}+h_{j,1}^{(n)}\sim_{\Tr}h_j$ for all $n$. So, applying Lemma \ref{lem:almostcommute},
\[
e^{ih_j}=e^{ih_{j,0}^{(n)}}e^{ih_{j,1}^{(n)}}e^{ic_n},
\]
where $c_n\sim_{\Tr}0$ for all $n$ and $c_n\to 0$. Next, using
 Lemma \ref{groupcommutators},
\begin{align*}
u &=\prod_{j=1}^k e^{ih_{j,0}^{(n)}}e^{ih_{j,1}^{(n)}}\cdot e^{ic_n'}\\
&=\prod_{j=1}^k e^{ih_{j,0}^{(n)}}\prod_{j=1}^k e^{ih_{k-j,1}^{(n)}}\prod_{j=1}^{k-1} (u_{j,n},v_{j,n})\cdot e^{ic_n'},
\end{align*}
for some $u_{j,n},v_{j,n}\in \U^0(A)$ and $c_n'\sim_{\Tr}0$ such that $c_n'\to 0$.
We will show that
$\prod_{j=1}^k e^{ih_{j,0}^{(n)}}$ is a product of $2k$ commutators, up to a multiplicative error of the form $e^{id_n}$ with $d_n\sim_{\Tr}0$ and $d_n\to 0$. The same arguments will apply to  $\prod_{j=1}^k e^{ih_{j,1}^{(n)}}$, thus proving the theorem.

Since the $r_{j,0}$'s are central in $N_0$, $\phi_0(r_{j,0})$ commutes with $x_{j',0,0}$ and $x_{j',0,1}$ for all $j'=2,\dots,k$. Hence, the lifts of these elements commute asymptotically. Applying Lemma \ref{lem:almostcommute} we have that
\begin{align*}
e^{ih_{j,0}^{(n)}} &=e^{[(x_{j,0,0}^{(n)})^*,x_{j,0,0}^{(n)}]}\cdot e^{[(x_{j,0,1}^{(n)})^*,x_{j,0,1}^{(n)}]}\cdot e^{iz_{j,0}^{(n)}}\cdot
e^{id_{j,n}},
\end{align*}
for all $j=2,\dots,k$, where $d_{j,n}\sim_{\Tr}0$ and $d_{j,n}\to 0$. So
\begin{align*}
\prod_{j=1}^k e^{ ih_{j,0}^{(n)} } &=e^{ih_{1,0}^{(n)}+\sum_{j=1}^k z_{j,0}^{(n)} }\cdot \Big(\prod_{j=2}^k e^{[(x_{j,0,0}^{(n)})^*,x_{j,0,0}^{(n)}]}e^{[(x_{j,0,1}^{(n)})^*,x_{j,0,1}^{(n)}]}\Big)\cdot e^{id_n}\\
&=e^{i[(y_{0,0}^{(n)})^*,y_{0,0}^{(n)}]}e^{i[(y_{0,1}^{(n)})^*,y_{0,1}^{(n)}]}\Big(\prod_{j=2}^k e^{[(x_{j,0,0}^{(n)})^*,x_{j,0,0}^{(n)}]}e^{[(x_{j,0,1}^{(n)})^*,x_{j,0,1}^{(n)}]}\Big)\cdot e^{id_n'}.
\end{align*}
By Lemma \ref{lem:x2},  the unitaries of the form  $e^{i[z^*,z]}$  in the above expression are all expressible as commutators,  up to a small multiplicative error of the form $e^{id}$,  with $d\sim_{\Tr}0$. This completes the proof.
\end{proof}

\begin{remark}
When considering the possibility of extending Theorem \ref{thm:prodexp} to C*-algebras of finite nuclear dimension $m>1$, one stumbles with the fact that the elements of a partition of unity $e_0,e_1,\dots,e_m$
associated to a decomposition of $\iota\colon A\to A_{\infty}$ may not pairwise commute. Thus, working in $A_\infty$, we cannot guarantee that for a selfadjoint $h\in A$ we have
\begin{align}\label{exppartition}
e^{ih}=\prod_{j=0}^m e^{ihe_j}.
\end{align}
However, once this obstruction is lifted, it is in fact possible to make the proof of Theorem \ref{thm:prodexp} go through. Notice that in order for \eqref{exppartition} to hold, it is not necessary to ask that all the $e_j$s
pairwise commute, but rather they the set $\{e_0,e_1,\dots,e_m\}$ can be recursively partitioned into pairs of subsets, so that the sums of the $e_j$s in these subsets commute. More specifically, the elements element $e_0,\dots,e_m$ are assigned in a bijective fashion  to the set of leaves of some binary tree. Then, for any node of the tree $\lambda$, with children $\lambda_0$ and $\lambda_1$, we require that the elements $\sum_{j\in \Lambda_0} e_j$ and $\sum_{j\in \Lambda_1} e_j$ commute, where $\Lambda_0$ is the set of leaves under $\lambda_0$, and $\Lambda_1$ the leaves under $\lambda_1$. Let us say that the  elements $e_0,\dots,e_m$ \emph{recursively commute} if the have this property for some binary tree.
Here is the statement of the resulting theorem:
\end{remark}
\begin{theorem}
Let $A$ be a C*-algebra of nuclear dimension $m\in \N$. Suppose that $A$ has a partition of unity $\{e_0,\dots,e_m\}$ whose elements recursively commute (in the sense described above). Let $u\in \U^0(A)$ be of the form
$u=e^{i h_1}e^{ih_2}\cdots e^{i h_k}$,
where $h_j\in A_{\sa}$ for $j=1,\dots,k$ and
\[
\sum_{j=1}^kh_j\sim_{\mathrm{Tr}}0.
\]
Then for each $\epsilon>0$
there exist unitaries $v_j,w_j\in \U^0(A)$, with $j=1,\dots,M$ and $M=2(m+1)k+\lfloor \frac {m+1} 2\rfloor(k-1)$, such that
\[
u=\prod_{j=1}^{M}(v_j,w_j)\cdot e^{ic},
\]
for some $c\in A_{\sa}$ with $c\sim_{\Tr}0$ and $\|c\|<\epsilon$.
\end{theorem}
 \begin{proof}
 Let $h_{j,l}=e_lh_j$ for $j=1,\dots,k$ and $l=0,\dots,m$. By the recursive commutation relations of the $e_l$'s, we have
$ e^{ih_j}=\prod_{l=0}^me^{ih_{j,l}}$. So
 \begin{align}\label{prodprodexp}
 \prod_{j=1}^k e^{ih_j}=\prod_{j=1}^k \prod_{l=0}^m e^{ih_{j,l}}.
 \end{align}
 We use the following extension of Lemma \ref{groupcommutators}:
 \[
 \prod_{j=1}^k\prod_{l=0}^m a_{j,l} =\prod_{j=1}^ka_{j,0}\prod_{j=0}^{k-1} a_{k-j,1}\prod_{j=1}^ka_{j,2}\cdots \prod_{i=1}^{\lfloor \frac {m+1} 2\rfloor(k-1)} (u_i,v_i).
 \]
 Thus, applying this formula to \eqref{prodprodexp}, 
 \[
 \prod_{j=1}^k e^{ih_j}= \prod_{j=1}^k e^{ih_{j,0}}\prod_{j=0}^{k-1} e^{ih_{k-j,1}}\cdots\prod_{i=1}^{\lfloor \frac {m+1} 2\rfloor(k-1)} (u_i,v_i)
 \]
 But for each $l=0,\dots,m$  the product of exponentials $\prod_{j=1}^k e^{ih_{j,l}}$ is expressible as a product of $2k$ commutators (as argued in the proof of Theorem \ref{thm:prodexp}). This shows that
 $\prod_{j=1}^k e^{ih_j}$ is a product of $M=2(m+1)k+\lfloor \frac {m+1} 2\rfloor(k-1)$ commutators in $A_\infty$. Furthermore, arguing as in the proof of Theorem \ref{thm:prodexp}, it is possible to lift these commutators so that the same expression holds in $A$, up to a small multiplicative error of the form $e^{ic}$, with $c\sim_{\Tr}0$. The details are left to the reader.
 \end{proof}

\section{C*-subalgebra of nuclear dimension one in a simple pure C*-algebra of stable rank one}\label{subalgebra}
In this section, we rely on  classification results using the Cuntz semigroup functor to establish the following theorem:
\begin{theorem}\label{embedding}
Let $A$ be a simple,  $\sigma$-unital, pure, C*-algebra of stable rank one. Then there exists a C*-subalgebra $B\subseteq A$ with the following properties:
\begin{enumerate}[(i)]
\item
$B\cong C\otimes \mathcal R$, where $C$ is a simple AF C*-algebra and $\mathcal R$ is the Jacelon-Razak C*-algebra (see \cite{jacelon}).
\item
Every lower semicontinuous trace on $B$ extends uniquely to a lower semicontinuous 2-quasitrace on $A$. Furthermore, for bounded traces the norm is preserved.
\item
Every selfadjoint $h\in A_{\sa}$ with connected spectrum is approximately unitarily equivalent to some $h'\in B_{\sa}$.
\end{enumerate}
\end{theorem}
Before proving the theorem, we make some preliminary remarks on the Cuntz semigroup of pure C*-algebras.
The reader is referred to \cite{ara-perera-toms} for the fundamental facts on the Cuntz semigroup of a C*-algebra.

Let  $A$ be a C*-algebra
satisfying the hypotheses of Theorem \ref{embedding}.
The Cuntz semigroup of $A$ decomposes into two subsemigroups:
\begin{enumerate}
\item
The set $\Cu_{\mathrm c}(A)$ of non-zero compact elements  (i.e., $[a]\in \Cu(A)\backslash\{0\}$ such that $[a]\ll [a]$) is a subsemigroup of $\Cu(A)$.
\item
The complement of $\Cu_{\mathrm c}(A)$ is also a subsemigroup which we denote by $\Cu_{\mathrm{nc}}(A)$.
\end{enumerate}
Since $A$ is stably finite,  $\Cu_{\mathrm c}(A)$ is isomorphic to  the semigroup  $\mathrm V^*(A)$  of non-zero Murray-von Neumann equivalence classes of projections (by \cite{brown-ciuperca}). On the other hand, by \cite[Theorem 5.6]{antoine-bosa-perera},  $\Cu_{\mathrm{nc}}(A)$ is isomorphic to $\mathrm{LAff}(\mathrm{QT}_2(A))$,
where $\mathrm{QT_2(A)}$ is the cone of densely finite lower semicontinuous 2-quasitraces on $A$ and $\mathrm{LAff}(\mathrm{QT}_2(A))$ is a suitable space of affine lower semicontinuous functions on $\mathrm{QT}_2(A)$.
We thus have
\begin{align*}
\Cu(A)&=\Cu_c(A)\sqcup \Cu_{\mathrm{nc}}(A)\\
&\cong \mathrm V^*(A)\sqcup \mathrm {LAff}(\mathrm{QT}_2(A)),
\end{align*}
where $\mathrm V^*(A)\sqcup \mathrm {LAff}(\mathrm{QT}_2(A))$ is endowed with a suitable order and addition (see \cite[Theorem 5.6]{antoine-bosa-perera}).

In \cite{nccw},
 the ordered semigroup $\Cu^\sim(A)$ is introduced and used as a classification invariant. For $A$ as before, it is shown in \cite[Section 6]{nccw} that
\[
\Cu^\sim(A)\cong \mathrm K_0(A)\sqcup \mathrm{LAff}^\sim(\mathrm{QT}_2(A)),
\]
where again $\mathrm{LAff}^\sim(\mathrm{QT}_2(A))$ is a suitable ordered semigroup
of lower semicontinuous functions on  $\mathrm{QT}_2(A)$. The reader is referred to
\cite{nccw} for its definition. We point out here that $\mathrm{LAff}(\mathrm{QT}_2(A))$ is the positive cone of $\mathrm{LAff}^\sim(\mathrm{QT}_2(A))$ and that $\mathrm{QT}_2(A)$ can be recovered from  $\mathrm{LAff}(\mathrm{QT}_2(A))$ (and thus from $\mathrm{LAff}^\sim(\mathrm{QT}_2(A))$) as the cone of positive functionals.

\begin{proof}[Proof of Theorem \ref{embedding}]
By \cite[Theorem 3.10]{blackadar}, there exists an AF C*-algebra $C$ such that $\mathrm{T}(C)\cong \mathrm{QT}_2(A)$. Tensoring with the Jacelon-Razak C*-algebra $\mathcal R$, we get a C*-algebra $B= C\otimes \mathcal R$ such that $\mathrm{K}_0(B)=0$ and
\[
\mathrm{T}(B)\cong \mathrm T(C)\cong \mathrm{QT}_2(A).
\]
 Thus,
\[
\Cu^\sim(B)\cong \mathrm{LAff}^\sim(\mathrm T(B))\cong \mathrm{LAff}^\sim(\mathrm{QT}_2(A)).
\]
(Up to this point, we have borrowed ideas from  \cite[Proposition 5.3]{jacelon} and \cite[Corollary 5.4]{jacelon}.)

Let $e_A\in A_+$ be a strictly positive element. Let $s_A\in A_+$ be such that
if $A$ is unital then $[s_A]$ is the maximal element strictly less than $[e_A]$, and if $A$
is non-unital then $[s_A]=[e_A]$. Notice that in either case $[s_A]\in \Cu_{\mathrm{nc}}(A)$. Let us choose $[s_B]\in \Cu(B)$ that is mapped to $[s_A]$ by the isomorphism $\Cu(B)\cong \Cu_{nc}(A)$. Passing to a stably isomorphic algebra if necessary, let us assume that $s_B\in B_+$ is a strictly positive element of $B$. We thus get an embedding $\Cu^\sim(B)\stackrel{\alpha}{\longrightarrow} \Cu^\sim(A)$ mapping $[s_B]$ to $[s_A]$.
Now observe that the C*-algebra $B$ is covered by the classification result of \cite{nccw}, since it is expressible as an inductive limit of 1-dimensional NCCW-complexes with trivial $K_1$-group. Thus, there exists a homomorphism $\phi\colon B\to A$ such that, after applying the functor $\Cu^\sim$, gives rise to $\alpha$. Since $\phi$ is necessarily an embedding, we will regard $B$ as a subalgebra of $A$.

It is clear, by construction, that $B$ satisfies (i). That it also satisfies (ii) follows from the fact that the inclusion $B\hookrightarrow A$ induces an isomorphism of
$\mathrm L(\mathrm {QT}_2(B))$
with $\mathrm L(\mathrm {QT}_2(A))$, whence, by duality,  of $\mathrm {QT}_2(A)$ with $\mathrm {QT}_2(B)=\mathrm {T}(B)$.

To prove (iii) we follow essentially the argument given in \cite{robert-commutators}: Say $h\in A_{\sa}$ has connected spectrum. Multiplying $h$ by a scalar if necessary, let us assume that $\|h\|\leq 1$.
Let $\psi\colon C_0(0,1]\oplus C_0(0,1]\to A$ be the homomorphism such that $\psi(t\oplus 0)=h_+$ and $\psi(0\oplus t)=h_-$. Since the spectrum of $h$ is connected, the map
$\Cu(\psi)$ ranges in $\Cu_{\mathrm{nc}}(A)$. Thus, it can be factored through $\Cu(B)$.
By the classification theorem \cite[Theorem 1]{ciuperca-elliott-santiago}, there exists
$\psi'\colon C_0(0,1]\oplus C_0(0,1]\to B$ such that $\Cu(\psi')=\Cu(\psi)$. Furthermore,
$\psi'$ and $\psi$ are approximately unitarily equivalent. Thus, $h':=\psi'(t\oplus -t)$
is approximately unitarily equivalent to $h$.
\end{proof}

\section{Approximation by commutators}
In this section we prove Theorems \ref{approximateversionI} and   \ref{approximateversionII}  (they are used in the proof of Theorem \ref{manyexp} from the introduction).

\begin{lemma}\label{splitlemma}
Let  $A$ be a simple  non-type I  C*-algebra.
Let $h\in A_{\sa}$. Then for each $\epsilon>0$ there exist
$\tilde h\in A_{\sa}$ and $x\in A$ such that
\begin{enumerate}[(i)]
\item
$h\approx_\epsilon \tilde h +[x^*,x]$,
\item
$\tilde h$ has connected spectrum,
\item
$\tilde h$ and $[x^*,x]$ commute, $x^2=0$,   and $\|\tilde h\|,\|x\|\leq \|h\|$.
\end{enumerate}
\end{lemma}
\begin{proof}
Let $a=h_+$ and $b=h_-$. Let us assume  that $\|a\|\geq \|b\|$ (the case $\|a\|<\|b\|$ is dealt with similarly). We first consider the case when $b\neq 0$.
Using functional calculus, we can slightly perturb $a$  (within any specified error and without changing its norm) so that $ae=\|a\|e$ for some non-zero $e\in A_+$. So let us assume that such an $e$ exists. Similarly, let us assume that $bf=\|b\|f$ for some non-zero $f\in A_+$. Since $A$ is simple,  we can find a non-zero $e'\in \her(e)_+$ such that
 $e'=y^*y$ and $yy^*\in\her(f)$ for some $y\in A$. Since $A$ is also non-type I, we can find  $e''\in \her(e')_+$
with  spectrum $[0,\|a\|]$.
Let $x\in A$ be such that $e''=x^*x$ and $xx^*\in \her(f)$. Set $\tilde h=h-[x^*,x]$.

Let us show that $\tilde h$ and $x$ have the desired properties. Properties (i) and (iii)
are clearly true. Let us prove (ii).
Since $ax^*x=\|a\|\cdot x^*x$ and the spectrum of $x^*x$ is $[0,\|a\|]$, the spectrum of
$a-x^*x$ is  $[0,\|a\|]$. Similarly, since $bxx^*=\|b\|\cdot xx^*$ and the spectrum of $xx^*$ is
$[0,\|a\|]$, the spectrum of  $b-xx^*$ is   $[\|b\|-\|a\|,\|b\|]$. Finally, the spectrum of $\tilde h=(a-x^*x)-(b-xx^*)$  is  $[0,\|a\|]\cup [-\|b\|,\|a\|-\|b\|]=[-\|b\|,\|a\|]$.

Let us sketch briefly how the above argument can be adapted in the case that $b=0$: As before, we assume that $ae=\|a\|e$ for some non-zero $e\in A_+$.
Next we find $x\in \her(e)$ such that $x^*x$ has spectrum $[0,\|a\|]$ and $x^2=0$. Finally, we set $\tilde h=h-[x^*,x]$.
\end{proof}

\begin{theorem}\label{approximateversionI}
Let $A$ be a separable,  simple, pure C*-algebra with stable rank one  and such that every  2-quasitracial state on $A$ is a trace. Let $h\in A_{\sa}$, with $h\sim_{\Tr} 0$ and $\|h\|<\frac \pi 2$.
Then for each $\epsilon>0$ there exist unitaries $v_j,w_j\in \U^0(A)$, with $j=1,2,3,4,5$ and
$c\in A_{\sa}$ such that
\[
e^{ih}=\prod_{j=1}^5(v_j,w_j)\cdot e^{ic},
\]
where $c\sim_{\Tr}0$, $\|c\|<\epsilon$, and $\|1-v_j\|,\|1-w_j\|\leq \|e^{ih}-1\|^{\frac 1 2}$ for all $j$.
\end{theorem}

\begin{proof}
Letting $\epsilon=\frac 1 n$, with $n=1,\dots$, in   the previous lemma, we find  $h_n\in A_{\sa}$ and $x_n\in A$ such that $h_n$
has connected spectrum, $x_n^2=0$, $h_n$ and $[x_n^*,x_n]$ commute, and
\[
\|h_n+[x_n^*,x_n]-h\|<\frac 1 n
\]
for all $n\in \N$. Let $B\subseteq A$ be the sub-C*-algebra whose existence was established in Theorem \ref{embedding}. Then,
 each $h_n$ is approximately unitarily equivalent to some $h_n'\in B$.
We have $|\tau(h_n')|=|\tau(h_n)|\leq \frac 1 n$ for any tracial state $\tau$ on $A$.
Our hypotheses imply that every tracial state on $B$ extends uniquely to a tracial state on $A$. Indeed, by Theorem \ref{embedding} (ii) every  tracial state on $B$ extends to a 2-quasitracial state on $A$, which by assumption must be a trace.
Thus, $|\tau(h_n')|<\frac 1 n$ for every tracial state $\tau$ on $B$. We can then perturb $h_n'$ (up to an error of $\frac 1 n$)  inside of $B$ to get $h_n''\in B_{\sa}$
such that $h_n''\sim_{\Tr}0$ (in $B$). We thus have that
$u_n^*h_n''u_n+[x_n^*,x_n]\to h$, where $u_n\in \U(A)$ are unitaries. By Lemma \ref{lem:almostcommute}  we have
\[
e^{ih}=e^{i[x_n^*,x_n]}u_n^*e^{ih_n''}u_ne^{ic_n},
\]
where $c_n\sim_{\Tr}0$ and $c_n\to 0$. Furthermore, by Lemma \ref{lem:x2},  $e^{i[x_n^*,x_n]}=(v_{n,0},w_{n,0})$ for some $v_{n,0},w_{n,0}\in \U^0(A)$. Thus,
\begin{align}\label{cojugateinB}
e^{ih}=(v_{n,0},w_{n,0})u_n^*e^{ih_n''}u_ne^{ic_n}.
\end{align}
Finally, applying
Theorem \ref{nuc1exp} in the algebra $B$ -- which has nuclear dimension 1 -- we get unitaries $v_{n,j},w_{n,j}\in \U^0(A)$, for $n=1,2,3,4$ such that
\[
e^{ih}=(v_{n,0},w_{n,1})\prod_{j=1}^4(v_{n,j},w_{n,j})e^{ic_n'}.
\]
where $c_n'\sim_{\Tr}0$ and $c_n'\to 0$. Lemmas \ref{splitlemma} and  \ref{lem:x2} and Theorem \ref{nuc1exp} also guarantee that the inequalities $\|v_{n,j}-1\|,\|w_{n,j}-1\|\leq \|e^{ih}-1\|^{\frac 1 2}$ hold for $j=0,1,2,3,4$ and all $n$.
\end{proof}

\begin{theorem}\label{approximateversionII}
Let $A$ be as in Theorem \ref{approximateversionI}.
Let $u\in \U^0(A)$ be of the form $u=\prod_{j=1}^k e^{ih_j}$, where $h_j\in A_{\sa}$ for all $j$, and such that $\Delta_{\Tr}(u)=0$.
Then for each $\epsilon>0$ there exist unitaries $v_j,w_j\in \U^0(A)$, with $j=1,\dots,7k+6$, such that
\[
u=\prod_{j=1}^{7k+6}(v_j,w_j)\cdot e^{ic}
\]
where $c\sim_{\Tr}0$ and $\|c\|<\epsilon$.
\end{theorem}

Before proving the theorem, we need some preparatory lemmas.
We let $\mathcal Z$ denote  the Jiang-Su C*-algebra.
\begin{lemma}
For each  non-zero $e\in \mathcal Z_+$ we can find  $h\in \her(e)_{\sa}$ such that $h\sim_{\Tr} 1$ and $e^{2\pi ih}=(u,v)$ for some unitaries $u,v\in \U^0(\her(e))$.
\end{lemma}
\begin{proof}
It suffices to prove the lemma for a sufficiently ``Cuntz small" element $e$. For suppose that we are given $e\in \mathcal Z_+$,  and that the lemma  has been proven for some $e'\in\mathcal Z_+ $ such that $e'=x^*x$ and $xx^*\in \mathrm{her}(e)$ for some $x\in \mathcal Z$. Say $h'\in \her(e')_{\sa}$ and $u',v'\in \U^0(\her(e'))$ are such that $h'\sim_{\Tr}1$ and $e^{2\pi i h'}=(u',v')$.
Let $w\in \mathcal Z^{**}$ be a partial isometry in the bidual of $\mathcal Z$ such that $x=w|x|$.
Then setting $h=wh'w^*$, $u=wu'w^*$ and $v=wvw^*$, the lemma is proven for
$e$.

Let $n\in \N$ and consider the dimension drop algebra $\mathcal Z_{n-1,n}$ as a  subalgebra of $\mathcal Z$. Let $e\in \mathcal Z_{n-1,n}$ be a positive element such that
$\mathrm{rank}(e(t))=n$ for $t\in (0,1]$ and $\mathrm{rank}(e(0))=n-1$.
 Then
$(n-1)[e]\leq [1]\leq n[e]$ in $\mathcal Z_{n-1,n}$ (whence in $\mathcal Z$). Therefore, by choosing $n$ large enough, we can arrange for  $[e]$ to be arbitrarily small in the Cuntz semigroup. By the argument in the previous paragraph, it suffices to prove the lemma for such an $e$.
The hereditary subalgebra in $\mathcal Z_{n-1,n}$
generated by $e$ is isomorphic to
\[
A_{n-1,n}=
\left \{\,
f\in \mathrm M_n(\mathrm C[0,1])\mid
\begin{array}{l}
f(0)=\mu 1_{n-1}\\
 f(1)=\lambda 1_n
\end{array},
\hbox{ for some }\lambda,\mu\in \C
\,\right \}.
\]
Let $\omega\colon [0,1]\to [0,1]$ be a continuous function such that $\omega(t)=0$ for $t\in [0,\frac 1 3]$,
$\omega$ is linear in  $[\frac 1 3,\frac 2 3]$, and $\omega(t)=1$ for $t\in [\frac 2 3,1]$.
Let
\[
h(t)=
\begin{pmatrix}
n-\omega(t) &   & &\\
& n-\omega(t) & &\\
& & \ddots & \\
 &&&(n-1)\omega(t)
\end{pmatrix}.
\]
Notice that $h\in A_{n-1,n}$ and that identifying $A_{n-1,n}$ with $\her(e)\subseteq \mathcal Z_{n-1,n}$ we have  $h\sim_{\Tr} 1\in \mathcal Z_{n-1,1}$. Hence, $h\sim_{\Tr} 1$ in $\mathcal Z$ after identifying $\mathcal Z_{n-1,n}$ with a unital subalgebra of $\mathcal Z$. We will be done once we have shown that $e^{2\pi i h}=(u,v)$, for some $u,v\in \U_0(A_{n-1,n})$.

We have
\[
e^{2\pi i h(t)}=
\begin{pmatrix}
e^{-2\pi i \omega(t)} &   & &\\
& e^{-2\pi i \omega(t)} & &\\
& & \ddots & \\
 &&&e^{2\pi i (n-1)\omega(t)}
\end{pmatrix}.
\]
Notice that $e^{ih(t)}$ is the identity matrix for $t\in [0,\frac 1 3]\cup [\frac 2 3,1]$ and that the elements on the diagonal of $e^{ih(t)}$ multiply to 1 for all $t\in [0,1]$. So $e^{ih(t)}=(P,v(t))$ for all $t\in [\frac 1 3,\frac 2 3]$, where $P$ is a suitable permutation unitary and $v(t)\in M_n$ is a diagonal unitary for all $t\in [\frac 1 3,\frac 2 3])$ (see by \cite[Lemma 2.1]{thomsen}). Furthermore, $v(\frac 1 3)=v(\frac 2 3)=1_n$. Thus, after extending $v$ to be constant equal to $1_n$ on $[0,\frac 1 3]\cup [\frac 2 3,1]$ and extending $P$ using a path that connects it to the identity matrix $1_n$
we get $u,v\in \U^0(A_{n-1,n})$ such that $e^{ih}=(u,v)$.
\end{proof}

\begin{lemma}
Let $A$ be a pure, simple, unital C*-algebra with stable rank one. Let $p\in M_n(A)$
be a projection in some matrix algebra over $A$. For each $e\in A_+$ we can find   a selfadjoint $h\in \her(e)_+$ such that  $h\sim_{\mathrm{Tr}}p$ and $e^{ih}=(v,w)$ for some $v,w\in U_0(A)$.
\label{lem:July25_2014_3PM}
\end{lemma}
\begin{proof}
By the classification theorem from \cite{nccw},
we can embed $\mathcal Z$ unitally in  $pM_n(A)p$.
Let us choose $e\in \mathcal Z_+$ such that $(n+1)[e]\leq [p]$ in $\mathcal Z$. By the previous lemma, there exists $h\in \mathrm{her}(e)_{\sa}\cap \mathcal Z$ such that $h\sim_{\Tr} p$ (inside $\mathcal Z$) and
$e^{ih}=(u,v)$ for some $u,v\in \U^0(\her(e)\cap \mathcal Z)$. Notice that  $h\sim_{\Tr}p$ also in  $M_n(A)$. Since $(n+1)[e]\leq [p]\leq n[1]$ and $\Cu(A)$ is almost unperforated, we have $[e]\leq [1]$ in $\Cu(A)$. Furthermore, $A$ has stable rank one, so there exists $x\in M_n(A)$ such that $e=x^*x$ and $xx^*\in A$.
Let $w\in \mathcal (M_n(A))^{**}$ be a partial isometry such that $x=w|x|$.
Setting $h=wh'w^*$, $u=wu'w^*$ and $v=wvw^*$, we get the desired elements in $A$.
\end{proof}

\begin{remark}
The previous proof can be adapted to  also apply to any $\mathcal Z$-stable C*-algebra.
\end{remark}

\begin{lemma}\label{commutatorinU0}
Let $A$ be a separable simple C*-algebra of stable rank one.
Let $u\in \U(A)$ and $h\in A_{\sa}$. For any $\epsilon>0$ there exist $v,w\in \U^0(A)$ such that
\[
(u,e^{ih})=(v,w)e^{ic},
\]
for some $c\in A_{\sa}$ such that $c\sim_{\Tr}0$ and $\|c\|<\epsilon$.
\end{lemma}
\begin{proof}
By functional calculus on $h$, we can find $e\in A_+$ such that $\|[h,x]\|\leq \delta\|x\|$ for any $x\in \her(e)$, where $\delta>0$ is small enough (how much to be specified soon). Since $A$ and $\her(e)$ are stably isomorphic and have stable rank one, $\U(A)/\U^0(A)\cong \mathrm K_1(A)\cong \U(\her(e))/\U^0(\her(e))$. Thus, there exists $v\in \U(\her(e))\subseteq \U(A)$ connected to $u$. Choosing  $\delta$ small enough, we can arrange that $\|\mathrm{Log}(ve^{ih}v^*e^{-ih})\|<\epsilon$, so that
$ve^{ih}v^*=e^{ih}e^{ic}$, for some $c\sim A_{\sa}$ with $c\sim_{\Tr}0$
and $\|c\|<\epsilon$.
Then,
\[
(u,e^{ih})=uv^*(ve^{ih}v^*)vu^*e^{-ih}=uv^*e^{ih}e^{ic}vu^*e^{-ih}=(uv^*,e^{ih})e^{ic'},
\]
where $c'\sim_{\Tr}0$
and $\|c'\|<\epsilon$. Since $uv^*\in \U^0(A)$, the lemma is proven.
\end{proof}

\begin{proof}[Proof of Theorem \ref{approximateversionII}] 
Since $\Delta_{\Tr}(u) = 0$, there exist  $n \geq \mathbb{N}$ and projections
$p, q \in M_n(A)$ such that 
$$\sum_{j=1}^k h_j + p - q \sim_{\Tr} 0.$$
Hence, by Lemma \ref{lem:July25_2014_3PM}, there exist $h_{k+1} \in A_{sa}$
and $u_{0}, v_{0} \in \U^0(A)$ such that
\[
u = \prod_{j=1}^{k+1} e^{i h_j} (u_{0}, v_{0})
\]
and 
\[
\sum_{j=1}^{k+1} h_j \sim_{\Tr} 0.
\]

As in the proof of
Theorem \ref{approximateversionI} (see \eqref{cojugateinB}), we find that for each $j=1,\dots,k+1$ there exist
$u_j\in \U(A)$, $h_j'\in B_{\sa}$, and $v_j,w_j\in \U^0(A)$ such that
\[
e^{ih_j}= (v_j,w_j)u_j^*e^{ih_j'}u_je^{ic_j},
\]
for some $c_j\in A_{\sa}$ such that $c_j\sim_{\Tr}0$ and $\|c_j\|$ is small. We thus get that
\begin{align*}
u
&=\prod_{j=1}^{k+1} u_j^*e^{ih_j'}u_j \prod_{j=1}^{k+2} (v_j',w_j')e^{ic}\\
&=u_1^*(\prod_{j=1}^{k+1} e^{i h_j'})u_1\prod_{j=1}^{k} (v_j'',w_j'')
\prod_{j=1}^{k+2} (v_j',w_j')e^{ic'}.
\end{align*}
Here we have used that  $(e^{ih_j'},u_j)=(v_j'',w_j'')e^{ic_j'}$, with $v_j'',w_j''\in \U^0(A)$,
$c_j'\sim_{\Tr}0$ and $\|c_j'\|$ small, by Lemma \ref{commutatorinU0}.
Since $B$ has nuclear dimension one, we can apply Theorem \ref{thm:prodexp} to $\prod_{j=1}^{k+1} e^{i h_j'}\in \U^0(B)$:
\[
\prod_{j=1}^{k+1} e^{i h_j'}=\prod_{j=1}^{5k +4} (v_j''',w_j''')e^{ic''},
\]
with $c''\sim_{\Tr}0$ and $\|c''\|$ small. Hence
\[
\prod_{j=1}^k e^{ih_j}=\prod_{j=1}^{5k+4} (v_j''',w_j''')\prod_{j=1}^{k}
(v_j'',w_j'')\prod_{j=1}^{k+2} (v_j',w_j')e^{ic'''},
\]
where $c'''\sim_{\Tr}0$. Choosing the elements $c_j$, $c_j'$ and $c''$ along the way small enough we can arrange for $\|c'''\|<\epsilon$. This proves the theorem.
\end{proof}

\section{Exact product of commutators}
Here we prove  Theorem \ref{singleexp} below and then use it in combination
with the results from the previous section to prove Theorem \ref{manyexp} from the introduction.

\begin{theorem}\label{singleexp}
There exist  $\delta>0$ and $C>0$ such that if $A$ is  a separable, simple,  pure, C*-algebra with stable rank 1 and such that every  2-quasitracial state on $A$ is a trace, and $u\in \U^0(A)$ is such that
$\|u-1\|<\delta$, and $\mathrm{Log}(u)\sim_{\Tr}0$, then
\begin{align}\label{23commutators}
u=\prod_{i=1}^{23} (v_i,w_i),
\end{align}
where $v_i,w_i\in \U^0(A)$ are such that
\begin{align}\label{23totheunit}
\|v_i-1\|,\|w_i-1\|\leq C\|e^{ih}-1\|^{\frac 1 2}
\end{align}
for all $i=1,\dots,23$.
\end{theorem}

Before proving Theorem \ref{singleexp} we need a number of lemmas.
Throughout this section $A$ denotes a C*-algebra that is separable,  simple, pure, of stable rank 1, and whose bounded 2-quasitraces are traces.

\begin{lemma}\label{2noninvertibles}
Let $h\in A_{\sa}$ be such that $h\sim_\Tr 0$ and $\|h\|<\frac \pi 2$.
Then there exist $u,v\in \U^0(A)$ and $h_1,h_2\in A_{\sa}$ such that $h_1,h_2\sim_{\Tr} 0$,
\[
e^{i h}=(u,v)e^{ih_1}e^{ih_2},
\]
$h_1$ and $h_2$ are non-invertible, $\|u-1\|,\|v-1\|\leq \|e^{ih}-1\|^{\frac 1 2}$, and $\|h_j\|\leq 2\|h\|$ for $j=1,2$.
\end{lemma}

\begin{proof}
We will assume that $A$ is unital and $h$ is invertible, since otherwise there is nothing to prove. Let  $a=h_+$ and $b=h_-$. Since $A$ is unital, stably finite, and its bounded 2-quasitraces are traces, there is at least one non-trivial bounded trace on $A$. This implies that $a\neq 0$ and $b\neq 0$.
We claim that there exists $x\in \her(a)$  such that $x^2=0$, $\|x\|\leq \|a\|$, $a$ and $[x^*,x]$ almost commute, and $a-x^*x+xx^*-b$ is not invertible. Before proving the existence of such an $x$, let us show how it suffices to prove the lemma: By Lemma \ref{lem:x2} (ii), there exist $u,v\in \U^0(A)$ such that $e^{i[x^*,x]}=(u,v)$. Also,
if $a$ and $[x,x^*]$ commute sufficiently, then by Lemma \ref{lem:almostcommute} there exists $c\in \her(a)_{\sa}$ such that $c\sim_{\Tr}0$ and 
$e^{i[x,x^*]}e^{ia}=e^{ic}e^{i[x,x^*]+ia}$.
Thus,
\begin{align*}
e^{ih} &=(u,v)e^{i[x,x^*]}e^{ia}e^{-ib}\\
&=(u,v)e^{ic}e^{i[x,x^*]+ia}e^{-ib}\\
&=(u,v)e^{ic}e^{i([x,x^*]+a-b)}.
\end{align*}
Since $bc=0$ and $b\neq 0$,  $c$ is non-invertible. By assumption
$[x,x^*]+a-b$ is non-invertible, and clearly $[x,x^*]+a-b\sim_{\Tr}a-b\sim_{\Tr}0$.
Let us set $h_1=c$ and $h_2=[x,x^*]+a-b$, so that $e^{i h}=(u,v)e^{ih_1}e^{ih_2}$, with $h_1$ and $h_2$ non-invertible. 
The norm estimates in the statement of the lemma are easily checked. The lemma is thus proved.

Let us prove the existence of $x$ as above. Let $\epsilon>0$ (which will measure how much $a$ and $[x^*,x]$ commute). Let us first assume that the spectrum of $a$ has at least two distinct non-zero points  $0<t_1<t_2\leq \|a\|$.
Let us choose functions $g_1,g_2\in C_0(0,1]_+$ supported in  small intervals around $t_1$ and $t_2$ respectively, so that the
nonzero elements $e_1=g_1(a)$ and $e_2=g_2(a)$ satisfy that $e_1\perp e_2$ and $\|[a,y]\|\leq \epsilon \|y\|$ for any
$y\in \her(e_i)$ and $i=1,2$.
Now, since $A$ is simple, there exists $y\in A$
such that $e_2ye_1$ is nonzero. Applying some functional calculus to $|e_2ye_1|\in \her(e_1)_+$, we can find non-zero contractions $f,f'\in \her(e_1)_+$ and $x\in A$ such that
$f'f=f$, $a^{\frac 1 2}fa^{\frac 1 2}=x^*x$, and $xx^*\in \her(e_2)$.
Notice that $x^2=0$, $\|x\|\leq \|a\|$, and $[x^*,x]$ and $a$ commute within an error of $2\epsilon\|a\|$.
Let us  show that $a-x^*x+xx^*-b$ is not invertible. Recall that we have assume that $h=a-b$ is invertible. This implies that $a$ is invertible in $C^*(a)$. Let $d\in C^*(a)$ denote the inverse of $a^{1/2}$ inside $C^*(a)$. Then
\begin{align*}
f'd(a-x^*x+xx^*-b) &=f'a^{1/2}-f'fa^{\frac 1 2}+ f'dxx^*\\
&=f'dxx^*\\
&=0.
\end{align*}
Here we have used that $f'd\in \her(e_1)$, $xx^*\in \her(e_2)$,  and $e_1\perp e_2$. Thus, $f'd\perp a-x^*x+xx^*-b$, while $f'd\neq 0$. Thus, 
 $a-x^*x+xx^*-b$ is not invertible in $A$.

If $a$ has only a non-zero point in its spectrum then it is a scalar multiple of a projection.
In this case, it suffices to choose $x\in \her(a)$ such that $x^2=0$ and $\|x\|=\|a\|$.
The details are left to the reader.
\end{proof}

\begin{lemma}
Let $h\in A_{\sa}$ be non-invertible. Then there exists $e\in A_+$ such that $h\in \her(e)$ and 0 is not an isolated point of the spectrum of $e$.
\end{lemma}
\begin{proof}
 If 0 is not an isolated point of the spectrum of $|h|$, we are done.
Otherwise, let $p$ be the support projection of $|h|$. Since $h$ is not invertible we have $p<1$. So $(1-p)A(1-p)$ is a non-type I simple C*-algebra, which implies that there exists a positive element $e'\in(1-p)A(1-p)$ such that 0 is not an isolated
point of its spectrum. Then $e=p+e'$ has the desired property.
\end{proof}

The following lemma is a consequence of \cite[Lemma 5.17]{dlHarpe-Skandalis2} and \cite[Proposition 5.18]{dlHarpe-Skandalis2}.

\begin{lemma}\label{lem:swingshrink}
Let $e,f\in A_+$ be such that $e\perp f$ and $\overline{eA}\cong \overline{fA}$ as right Hilbert $A$-modules.
\begin{enumerate}[(i)]
\item
Let $g\in \her(e)_{\sa}$. Then there exist $u,v\in \her(e+f)$ and  $g'\in \her(f)_{\sa}$ such that $e^{ig}=(u,v)e^{ig'}$, with $\|u-1\|,\|v-1\|\leq \|e^{ig}-1\|$, $g'\sim_{\Tr}g$ and
$\|g'\|=\|g\|$.

\item
Let $h\in \her(e+f)_{\sa}$ be such that  $h\sim_{\Tr}0$  and $\|e^{ih}-1\|<\frac 1 4$. Then there exist
$u_i,v_i\in \U^0(\her(e+f))$, with $i=1,2$, and $h'\in \her(e)_{\sa}$, such that
\[
e^{ih}=(u_1,v_1)(u_2,v_2)e^{ih'}
\]
and
\begin{align*}
\|e^{ih'}-1\| &\leq 8\|e^{ih}-1\|,\\
\|u_1-1\|,\|v_1-1\|&\leq 2\|e^{ih}-1\|^{\frac 1 2},\\
\|u_2-1\|,\|v_2-1\|&\leq 2\|e^{ih}-1\|^{\frac 1 2}.
\end{align*}
Furthermore, if $\|h\|$ is small enough then $h'\sim_{\Tr}0$.
\end{enumerate}
\end{lemma}
\begin{proof}
Let  $p,q\in \mathrm M(\her(e+f))$  be the multiplier support projections of $e$ and $f$ respectively. The hypothesis that $\overline{eA}\cong \overline{fA}$ implies that they are
Murray-von Neumann equivalent in $\mathrm M(\her(e+f))$. Now (i) follows by a  direct application of
\cite[Lemma 5.17]{dlHarpe-Skandalis2} while (ii) follows from  \cite[Proposition 5.18]{dlHarpe-Skandalis2}.
\end{proof}

Before proving the next lemma, we introduce some notation.
Let $e\in A_+$ be such that
$0$ is not an isolated point of the spectrum of $e$. Then $[e]$ is a non-compact element of $\Cu(A)$.
By the computation of $\Cu(A)$ (see Section \ref{subalgebra}), there exist Cuntz classes $[e_1],[e_2],\dots$
such that $[e]=2[e_1]$ and $[e_n]=2[e_{n+1}]$ for all $n\in \N$. Since $A$ has stable rank one, Cuntz equivalence is the same as Hilbert module isomorphism (by \cite[Theorem 3]{coward-elliott-ivanescu}). Using this, we can choose the representatives $e_n\in A_+$ of these Cuntz classes such that
\begin{enumerate}
\item
$\overline{eA}=\overline{e_1A+e_2A+\dots}$,
\item
$e_n\perp e_m$ for all $n,m\geq 1$ with $n\neq m$.
\end{enumerate}
By construction, we also have the (right) Hilbert $A$-module isomorphisms
\begin{enumerate}
\item
$\overline{eA}\cong\overline{e_1A}\oplus \overline{e_1A}$,
\item
$\overline{e_nA}\cong \overline{e_{n+1}A}\oplus \overline{e_{n+1}A}$ for all $n\in \N$.
\end{enumerate}

In proving the following lemma we use   a technique that goes back to \cite{fack}.
We let $e,e_1,e_2,\ldots\in A_+$ be as in the previous paragraph.

\begin{lemma}\label{9commutators}
There exist  $\delta>0$ and $C>0$ (not depending on the C*-algebra $A$) such that
if $h\in \her(e)_{\sa}$ is such that $h\sim_{\Tr}0$ and $\|h\|<\delta$, then
\[
e^{ih}=\prod_{j=1}^{11} (v_j,w_j),
\]
for some  $v_j,w_j\in \U^0(A)$, $j=1,\dots,11$, such that
\begin{align}\label{totheunit}
\|v_j-1\|,\|w_j-1\|\leq C\|e^{ih}-1\|^{\frac 1 2}
\end{align}
for all $j=1,\dots,11$.
\end{lemma}
\begin{proof}
By Lemma \ref{lem:swingshrink} (ii),  there exist unitaries $u_1,v_1,u_2,v_2\in \U^0(A)$ such that
\begin{align}\label{2plus9}
e^{ih}=(u_1,v_1)(u_2,v_2)e^{ih_1},
\end{align}
where $h_1\in \her(e_1)_{\sa}$ and $h_1\sim_{\Tr}0$. We will show that $e^{ih_1}$ is a product of 9 commutators.

For $n=1,2,\dots$ let $e_{n+1}'\in \her(e_{n})_+$ be a positive element such that $\overline{e_{n+1}'A}\cong \overline{e_{n+1}A}$.

By Theorem \ref{approximateversionI} applied in $\her(e_1)$, there exist unitaries
$u_i^{(1)},v_i^{(1)}$, with $i=1,\dots,5$, in $\U^0(\her(e_1))$ such that
\[
e^{ih_1}=\prod_{i=1}^5 (u_i^{(1)},v_i^{(1)})e^{ih_1'},
\]
with $h_1'\in\her(e_1)_{\sa}$, $h_1'\sim_{\Tr}0$, and $\|h_1\|'<\frac{1}{2^1}$. Next, by Lemma \ref{lem:swingshrink} (ii), there exist unitaries $u_6^{(1)},v_6^{(1)},u_7^{(1)},v_7^{(1)}$ in $\U^0(\her(e_1))$ such that
\[
e^{ih_1'}=(u_6^{(1)},v_6^{(1)})(u_7^{(1)},v_7^{(1)})e^{ih_1''},
\]
and $h_1''\in \her(e_2')_{\sa}$ and $h_1''\sim_{\Tr}0$.  Finally,
 by Lemma \ref{lem:swingshrink} (i), we have
\[
e^{ih_1'}=(y^{(1)},z^{(1)})e^{ih_2},\]
with $y^{(1)},z^{(1)}\in \U^0(\her(e_1+e_2))$ and $h_2\in \her(e_2)_{\sa}$.
Next, we apply again Theorem \ref{approximateversionI} in $\her(e_2)$:
\[
e^{ih_2}=\prod_{i=1}^5 (u_i^{(2)},v_i^{(2)})e^{ih_2'},
\]
with $h_2'\sim_{\Tr}0$ and $\|h_2'\|<\frac{1}{2^2}$, followed by Lemma \ref{lem:swingshrink} (ii),
\[
e^{ih_2'}=(u_6^{(2)},v_6^{(2)})(u_7^{(2)},v_7^{(2)})e^{ih_2''},\]
with $h_2''\in \her(e_3')_{\sa}$ and $h_2''\sim_{\Tr}0$, and  Lemma \ref{lem:swingshrink} (ii):
\[
e^{ih_2''}=(y^{(2)},z^{(2)})e^{ih_3},
\]
with $h_3\in \her(e_3)_{\sa}$ and $h_3\sim_{\Tr}0$. Continuing this strategy, we construct for each $n\in \N$ unitaries $u_1^{(n)},v_1^{(n)},\dots,u_7^{(n)},v_7^{(n)}\in \U^0(\her(e_n))$,  $y^{(n)},z^{(n)}\in \U^0(\her(e_n+e_{n+1})$, and a selfadjoint $h_n\in \her(e_n)_{\sa}$,  such that  $h_n\sim_{\Tr}0$, $h_n\to 0$, and
\[
e^{ih_1}=\prod_{k=1}^{n-1}\Big(\prod_{j=7}^j(u_j^{(k)},v_j^{(k)})\Big) (y^{(k)},z^{(k)})\cdot e^{ih_n}.
\]
Since $h_n\to 0$, the above formula yields $e^{ih_1}$ expressed as an infinite product of commutators.
By the pairwise orthogonality of the $e_n$s, we can gather the terms of this infinite product into subsequences, each of them equal to a finite  product of commutators, as follows:
\begin{enumerate}
\item
$\prod_{k=1}^{\infty}\prod_{j=7}^j(u_j^{(k)},v_j^{(k)})$ is a product of 7 commutators,
\item
$\prod_{k=1}^\infty (y^{(2k-1)}),z^{(2k-1)})$ is a single commutator,

\item
$\prod_{k=1}^\infty (y^{(2k)}),z^{(2k)})$ is a single commutator.
\end{enumerate}
We thus arrive at an expression of $e^{ih_1}$ as a product of 9 commutators. This, together with \eqref{2plus9},  yields an expression of $e^{ih}$ as a product of 11  commutators. The estimates \eqref{totheunit} are straightforwardly derived from the construction of the unitaries involved in the commutators.
\end{proof}

\begin{proof}[Proof of Theorem \ref{singleexp}]
Let $h\in A_{\sa}$ be such that $\|h\|<\frac \pi 2$. Then by Lemma \ref{2noninvertibles}, we have $e^{ih}=(u,v)e^{ih_1}e^{ih_2}$, where $h_1,h_2\in A_{\sa}$ are non-invertible and $h_1,h_2\sim_{\Tr}0$. (If $A$ is non-unital this step is unnecessary). But $e^{ih_1}$ and $e^{ih_2}$ are both products of 11 commutators, by Lemma \ref{9commutators}. This yields \eqref{23commutators}. The verification of the norm estimates \eqref{23totheunit} is left to the reader.
\end{proof}

\begin{proof}[Proof of Theorem \ref{manyexp}]
By Theorem \ref{approximateversionII}, we have that
\[
u=\prod_{j=1}^{7k+6}(v_j,w_j)\cdot e^{ic} ,\]
where $v_j,w_j\in \U^0(A)$ and $c\in A_{\sa}$, with $c\sim_{\Tr}0$ and $\|c\|$  as small as desired. Thus, we can arrange for Theorem \ref{singleexp} to be applicable, so that $e^{ic}$ is expressed as a product of 23 commutators. In this way,   $u$ is expressed as a product of $7k+6+23=7k+29$ commutators.
\end{proof}

\begin{corollary} \label{Aug}
Let $A$ be as before. Then
\[
\mathrm{DU}^0(A)=\mathrm{DU}(A)=\ker(\Delta_{\Tr}).
\]
\end{corollary}
\begin{proof}
The equality $\mathrm{DU}^0(A)=\ker(\Delta_{\Tr})$ is a corollary of the previous theorem together with the fact that every element of $\U^0(A)$ is a finite product of exponentials.

To prove $\mathrm{DU}(A)=\ker(\Delta_{\Tr})$, it suffices to show that $\mathrm{DU}(A)\subseteq \ker(\Delta_{\Tr})$ (since we already know that $\mathrm{DU}^0(A)=\ker(\Delta_{\Tr})$).
In general,
\[
\mathrm{DU}(A)\cap \U^0(A)\subseteq \ker(\Delta_{\Tr}).\]
To see this, notice that elements of $\mathrm{DU}(A)$ are in $\mathrm{DU^0}(M_3(A))$, since
\[
\mathrm{diag}((u,v),1,1)=(\mathrm{diag}(u,u^*,1),(\mathrm{diag}(v,1,v^*)).
\]
Hence, if an element of $\mathrm{DU}(A)$ is in $\U^0(A)$, then its de la Harpe-Skandalis determinant evaluates to 0.  Since in our case $A$ is assumed to have stable rank one,  $\mathrm{DU}(A)\subseteq \U^0(A)$ by $\mathrm K_1$-injectivity. Thus, $\mathrm{DU}(A)\subseteq \ker(\Delta_{\Tr})$.
\end{proof}

\begin{theorem}\label{GLcase}
Let $A$ be a separable, simple, pure C*-algebra with stable rank 1 and such
that every 2-quasitracial state on $A$ is a trace.
Let $x \in \mathrm{GL}^0(A)$.

If $\Delta_{\Tr}(x) = 0$ then $x$ is a finite product of multiplicative
commutators in $\mathrm{GL}^0(A)$.
\end{theorem}

\begin{proof}[Sketch of proof.]
The proof is very similar to the proof in the unitary case.
We briefly indicate the modifications which need to be made.\\

\emph{Step 1:}
We establish the following:
There exists a number $k \geq 1$ such that
if $a \in A_{\sa}$ and $\Tr(a) = 0$ then
$e^a$ is a product of $k$ multiplicative commutators.

The proof of Step 1 is very similar to the proof of Lemma \ref{9commutators}. 
Firstly, we prove a version of Theorem \ref{approximateversionI}
for exponentials of the
form $e^c$ with $c \in A_{\sa}$ (as opposed to unitary exponentials).
The proof is exactly the same as that for Theorem \ref{approximateversionI}, 
except that in
the proof of Lemma 2.4, \cite{dlHarpe-Skandalis2} Lemma 5.13 
needs to be replaced with
\cite{thomsen} Lemma 2.6.  
Then we follow exactly the same line of argument as that
of Lemma \ref{9commutators}, except that Lemma \ref{lem:swingshrink} 
needs to be replaced with
\cite{dlHarpe-Skandalis2} 
Lemma 5.11 and Proposition 5.12.  (In other words, we are
actually following the line of argument of \cite{dlHarpe-Skandalis2} 
Proposition 6.1.)\\

\emph{Step 2:}
Now suppose that $x \in GL^0(A)$ with $\Delta_{\Tr}(x) = 0$.
Let $x = u |x|$ be the polar decomposition of $x$.
Then $\Delta_{\Tr}(u) = \Delta_{\Tr}(|x|) = 0$.

Since $\Delta_{Tr}(u) = 0$, by Theorem \ref{manyexp}, $u$ is a finite product of
commutators.
Since $|x|$ is a positive invertible, $\Delta_{\Tr}(|x|) = 0$ implies
that $\Tr(Log(|x|))=0$.  Hence, by Step 1, $|x|$ is a finite product of
commutators.  Hence, $x$ is a finite product of commutators, as required.

\end{proof}

\end{document}